\documentclass{amsart}

\newtheorem{theorem}{Theorem}[section]
\newtheorem{lemma}[theorem]{Lemma}
\newtheorem{proposition}[theorem]{Proposition}
\newtheorem{corollary}[theorem]{Corollary}

\newtheorem{conjecture}[theorem]{Conjecture}

\theoremstyle{definition}
\newtheorem{definition}[theorem]{Definition}
\newtheorem{example}[theorem]{Example}

\newtheorem{remark}[theorem]{Remark}

\def\cocoa{{\hbox{\rm C\kern-.13em o\kern-.07em C\kern-.13em o\kern-.15em A}}}

\usepackage{amscd,amssymb}
\usepackage{youngtab}

\newenvironment{proofof}[1]{\noindent{\it Proof of
#1.}}{\hfill$\square$\\\mbox{}}

\begin{document}

\title[Separating monomials for diagonalizable actions]
{Separating monomials for diagonalizable actions}

\author[M\'aty\'as Domokos]
{M\'aty\'as Domokos}
\address{Alfr\'ed R\'enyi Institute of Mathematics,
Re\'altanoda utca 13-15, 1053 Budapest, Hungary,
ORCID iD: https://orcid.org/0000-0002-0189-8831}
\email{domokos.matyas@renyi.hu}

\thanks{Partially supported by the Hungarian National Research, Development and Innovation Office,  NKFIH K 138828,  K 132002.}



\maketitle

\begin{abstract} 
Sets of monomials separating Zariski closed orbits under diagonalizable group actions are characterized in terms of 
the monoid of zero-sum sequences over the character group. 
This is applied to compare the degree bounds for separating invariants and generating invariants of diagonalizable group actions. 
\end{abstract}

\section{Introduction} \label{sec:intro}

Let $K$ be an algebraically closed field, and 
let $G$ be a diagonalizable linear algebraic group over $K$. 
It is well known that for a representation of $G$ on a finite dimensional $K$-vector space $V$, the algebra 
$\mathcal{O}(V)^G$ of $G$-invariant polynomial functions on $V$ is generated by monomials (with an appropriate 
choice of variables). The aim of this note is to characterize the sets of invariant monomials that form a so-called separating set in $\mathcal{O}(V)^G$. Recall that a subset $S\subset \mathcal{O}(V)^G$ is \emph{separating} if for any $v,w\in V$, such that $f(v)\neq f(w)$ for some $f\in \mathcal{O}(V)^G$, there exists an $h\in S$ with $h(v)\neq h(w)$ 
(see \cite[Section 2.4]{derksen-kemper} for this definition and some basic properties). 
The property of being a \emph{separating set} is obviously weaker than being a 
\emph{generating set}, 
but for several applications a separating set is just as good as a generating set. Therefore the study of separating sets has become popular in the past two decades. To give a sample, we mention \cite{draisma-kemper-wehlau}, \cite{grosshans:2007}, \cite{dufresne-elmer-kohls}, \cite{neusel-sezer}, \cite{kohls-kraft}, \cite{dufresne-jeffries}, \cite{kemper-lopatin-reimers}. 

For basic facts about diagonalizable linear algebraic groups see \cite[Section 8]{borel}. 
Denote by $X(G)$ the set of morphisms $G\to K^{\times}$ of algebraic groups; this is called the \emph{group of  characters of} $G$. Write $\mathbf{1}\in X(G)$ for the trivial character (i.e. $\mathbf{1}(g)=1$ for all $g\in G$). The coordinate ring $\mathcal{O}(G)$ of $G$ has $X(G)$ as a $K$-vector space basis, so it is isomorphic (as a Hopf algebra) to $K[X(G)]$, the group algebra of the abelian group $X(G)$. 

Throughout the paper we shall consider the following setup and notation. 
Given a representation of $G$ on a finite dimensional $K$-vector space $V$, we have a decomposition 
\begin{equation}\label{eq:decomposition of V} 
V=V_1\oplus\cdots\oplus V_n
\end{equation} 
as a direct sum of $1$-dimensional $G$-invariant subspaces $V_i$ of $V$. 
Pick a non-zero element $e_i\in V_i$. Then $e_1,\dots,e_n$ is a basis of $V$, and 
we have $g\cdot e_i=\chi_i(g)e_i$ for some $\chi_i\in X(G)$. 
Set 
\[\mathcal{B}(\chi_1,\dots\chi_n):=\{m\in \mathbb{N}_0^n\mid \prod_{i=1}^n\chi_i^{m_i}=\mathbf{1}\in X(G)\}.\] 
Obviously $\mathcal{B}(\chi_1,\dots\chi_n)$ is a submonoid of the additive submonoid $\mathbb{N}_0^n$ of the  free abelian group $\mathbb{Z}^n$. 
A non-zero element $m$ of $\mathcal{B}(\chi_1,\dots,\chi_n)$ 
is called an \emph{atom} if $m$ is not the sum of two non-zero elements of $\mathcal{B}(\chi_1,\dots,\chi_n)$.  
Write $\mathcal{A}(\chi_1,\dots,\chi_n)$ for the set of atoms in $\mathcal{B}(\chi_1,\dots,\chi_n)$.  
Now take a basis $x_1,\dots,x_n\in V^*$ dual to the basis $e_1,\dots,e_n$ of $V$. 
Then every monomial in $x_1,\dots,x_n$ spans a $G$-invariant subspace in $\mathcal{O}(V)$, hence $\mathcal{O}(V)^G$ is spanned by the $G$-invariant monomials. 
A monomial $x^m=x_1^{m_1}\cdots x_n^{m_n}$ is $G$-invariant if and only if 
$m\in   \mathcal{B}(\chi_1,\dots,\chi_n)$. 
Therefore  
$\{x^m\mid m\in \mathcal{B}(\chi_1,\dots,\chi_n)\}$ is a $K$-vector space basis of $\mathcal{O}(V)^G$. 
As a consequence of these observations we get the well-known statement below: 
\begin{proposition}\label{prop:monoid generators} 
For a subset $M\subset \mathcal{B}(\chi_1,\dots,\chi_n)$ the following conditions are equivalent: 
\begin{enumerate} 
\item The algebra $\mathcal{O}(V)^G$ is generated by 
$\{x^m\mid m\in M\}$. 
\item The monoid $\mathcal{B}(\chi_1,\dots,\chi_n)$ is generated by $M$. 
\item $\mathcal{A}(\chi_1,\dots,\chi_n)\subseteq M$. 
\end{enumerate}  
\end{proposition} 

\begin{remark} \label{remark:finitely generated} 
Not all submonoids of $\mathbb{N}_0^n$ are finitely generated. However, 
diagonalizable groups are linearly reductive (see for example \cite[8.4 Proposition]{borel}), 
and hence $\mathcal{O}(V)^G$ is finitely generated as a $K$-algebra. 
Consequently, $\mathcal{B}(\chi_1,\dots,\chi_n)$ is finitely generated as a monoid, or in other words, the set $\mathcal{A}(\chi_1,\dots,\chi_n)$ is finite. 
\end{remark} 

Proposition~\ref{prop:monoid generators} is the basis of a long known connection between the theory of zero-sum sequences over abelian groups (see \cite{gao-geroldinger} for a survey) and invariant theory of abelian groups. An early application was the deduction of the Noether number of certain finite abelian groups from known results on their Davenport constants in \cite{schmid}. For more information on the interplay between invariant theory, zero-sum theory, and factorization theory see \cite{cziszter-domokos-geroldinger} 
and the references therein. 

The first aim of this note is to find the analogues of Proposition~\ref{prop:monoid generators} for separating sets of invariant monomials.  
To state it we need some notation. 
To simplify earlier notation set $\mathcal{B}:=\mathcal{B}(\chi_1,\dots,\chi_n)$, 
$\mathcal{A}:=\mathcal{A}(\chi_1,\dots,\chi_n)$. 
For $m\in \mathbb{Z}^n$ write $\mathrm{supp}(m):=\{j\in \{1,\dots,n\}\mid m_j\neq 0\}$. 
Given a subset $J\subseteq \{1,\dots, n\}$, 
and a subset $Q\subseteq \mathbb{N}_0^n$ write 
$Q_J:=\{q\in Q\mid \mathrm{supp}(q)\subseteq J\}$,   
and denote by $\mathbb{Z}Q$ the additive subgroup of $\mathbb{Z}^n$ generated by $Q$. 
Note that $\mathcal{B}_J$ is a submonoid of $\mathcal{B}$, and $\mathcal{A}_J$ is the set of atoms in the monoid $\mathcal{B}_J$. 

\begin{theorem}\label{thm:main} 
Assume that $K$ has characteristic zero. 
Then the  following conditions are equivalent for a subset $M\subseteq \mathcal{B}$: 
\begin{enumerate} 
\item The monomials $\{x^m\mid m\in M\}$ form a separating set in $\mathcal{O}(V)^G$. 
\item For any subset $J\subseteq \{1,\dots,n\}$,  $\mathcal{A}_J$ is 
contained in $\mathbb{Z}M_J$. 
\end{enumerate} 
\end{theorem}

\begin{remark} 
(i) In the special case when $G$ is finite the above result was proved in \cite[Theorem 2.1]{domokos}. 
For the case when $G$ is a torus a closely related  
result was  given by Dufresne and Jeffries in 
\cite[Lemma 6.2, Proposition 6.3]{dufresne-jeffries:2}, but they provide a  
condition formally stronger than condition (2) in Theorem~\ref{thm:main} to guarantee that a set of invariant monomials is separating. We note that the arguments in loc. cit. work also for not necessarily connected diagonalizable groups.

(ii) We shall give two proofs of Theorem~\ref{thm:main}: first in Section~\ref{sec:1=2} we prove it using some basic facts about diagonalizable groups. In Section~\ref{sec:generalization} we state a more general version (Theorem~\ref{thm:general}) concerning separating sets in arbitrary subalgebras generated by monomials in the polynomial algebra, and deduce it from a basic algebro-geometric lemma. 
\end{remark}

\begin{definition}\label{def:tau(B)}
For a submonoid $B$ of the additive monoid $\mathbb{N}_0^n$ 
denote by $\tau(B)$ the minimal non-negative integer $t$ such that for all 
$I\subseteq \{1,\dots.n\}$, 
the abelian group $\mathbb{Z}B_I$ is generated by 
$\bigcup_{J\subseteq I\colon |J|\le t} B_J$. 
\end{definition} 

\begin{remark}\label{remark:2'} 
The motivation for Definition~\ref{def:tau(B)} is that 
condition (2) for $M$ in Theorem~\ref{thm:main} is obviously equivalent to 
\begin{enumerate} 
\item[(2')] For any subset $J\subseteq \{1,\dots,n\}$ with $|J|\le \tau(\mathcal{B})$, 
$\mathcal{A}_J$ is contained in $\mathbb{Z}M_J$. 
\end{enumerate}  
\end{remark}

For a finitely generated abelian group $X$ denote by $\mathrm{rk}(X)$ the minimal number of generators of $X$.  
Write $G^{\circ}$ for the connected component of the identity in $G$, so $G^\circ$ is a torus. By \cite[8.7 Proposition]{borel} we have 
$G\cong G^\circ\times A$; here $A$ is a finite abelian group (whose order is not divisible by $p$ when $\mathrm{char}(K)=p>0$). Therefore $X(G)=X(G^\circ)\times X(A)$, where 
$X(G^\circ)$ is a free abelian group of rank $\dim(G)$, and $X(A)\cong A$ is the torsion subgroup of $X(G)$. 
By Remark~\ref{remark:2'} the problem of effectively finding separating sets of monomials  is greatly simplified by the uniform bound (depending only on $G$) for $\tau(\mathcal{B})$ provided by our second main result: 
 
\begin{theorem}\label{thm:ZB_J}  
We have the inequality 
\[\tau(\mathcal{B})\le 1+2\dim(G)+\mathrm{rk}(X(G))=1+3\dim(G^\circ)+\mathrm{rk}(G/G^\circ).\]  
\end{theorem} 

\begin{remark} \label{remark:tau(B) for tori} 
(i) When $G=G^\circ$ is a torus, as a consequence of \cite[Theorem 6.7]{dufresne-jeffries:2} of Dufresne and Jeffries (see Remark~\ref{remark:dufresne-jeffries} in the present paper for the statement) and our Corollary~\ref{cor:tau=tau} of Theorem~\ref{thm:main}, the following improvement of Theorem~\ref{thm:ZB_J} holds: 
Assume that  $G=G^\circ$ is a torus. Then we have the inequality 
\[\tau(\mathcal{B})\le 1+\dim(G)+\mathrm{rk}(X(G))=1+2\dim(G).\]  

(ii) The monoid theoretic reformulation of Theorem~\ref{thm:ZB_J} independent of diagonalizable group actions will be discussed in Section~\ref{sec:monoid}, see Corollary~\ref{cor:reformulation}. 
\end{remark} 

For the case $\mathrm{char}(K)=p>0$ a characterization of separating sets of invariant monomials for torus actions is given by Dufresne and Jeffries \cite[Proposition 6.1]{dufresne-jeffries:2}, and their argument works for not necessarily connected diagonalizable groups. 
An extension of this result (in the flavour of Theorem~\ref{thm:main}) will be formulated in Theorem~\ref{thm:pos char main}.

It is natural to expect that the relaxation of the property of being a \emph{generating set} 
to being a \emph{separating set}  should be reflected in degree bounds even in the case of completely reducible actions (this is evidently so for the modular case, see 
\cite[Corollaries 3.3.4, 3.12.3]{derksen-kemper}, or \cite{kaygorodov-lopatin-popov}). However, as far as we know, there are not many proven results in the literature confirming this expectation. One such interesting example  is 
\cite[Theorem 1.14]{derksen-makam:0} about matrix invariants. 
A result relevant to this question on the non-abelian semidirect product $C_p\rtimes C_3$ can be found in 
\cite{cziszter}. 
For finite abelian groups it is shown in \cite[Corollary 3.11]{domokos} that typically separating sets exist in strictly smaller degree than generating sets. As an application of Theorem~\ref{thm:main} we prove  Corollary~\ref{cor:inf}, which is a result for diagonalizable groups of dimension $\ge 2$ that points to a similar direction. 

The connected diagonalizable groups (abelian connected reductive groups) are the tori. Their invariant theory (e.g. degree bounds) is studied in \cite{kempf}, \cite{wehlau}, \cite{wehlau:1994}. Motivated by questions of algebraic complexity, algorithmic aspects of torus actions are investigated in \cite{burgisser-etal}, where polynomial time algorithm is given  for the problems of orbit equality, orbit closure intersection, and orbit closure containment. See also \cite[Theorem 1.8]{garg_etal} for a result on the complexity of writing down generators of certain torus actions, and \cite[Section 1.5]{garg_etal} for further motivation to study separating invariants. 
Lower degree bounds for the generators of certain invariant rings of classical reductive groups are obtained from lower degree bounds for separating invariants of torus actions in \cite{derksen-makam:1}.  
Moreover, in \cite[Theorem 2.10]{derksen-makam:2} a criterion is given by which checking Zariski closedness of the orbit of a vector under the action of a reductive group is reduced to checking Zariski closedness of its orbits with respect to a family of tori.  

After collecting in Section~\ref{sec:prel}  some preliminary results on separating invariants of diagonalizable group actions, we prove Theorem~\ref{thm:main} in 
Section~\ref{sec:1=2}.    In Section~\ref{sec:monoid} we translate Theorem~\ref{thm:main} 
purely into the language of submonoids of $\mathbb{N}_0^n$. The material in Section~\ref{sec:helly} and 
Section~\ref{sec:proof of thm ZB_J} gives the proof of Theorem~\ref{thm:ZB_J}, 
and in Section~\ref{sec:conjectures} we state some conjectures related to these results. 
We turn to the version of Theorem~\ref{thm:main} valid in positive characteristic in Section~\ref{sec:positive characteristic}. The results on degree bounds for separating invariants versus generating invariants are contained in Section~\ref{sec:examples}. 
Finally in Section~\ref{sec:generalization} we generalize Theorem~\ref{thm:main} to 
arbitrary monomial subalgebras of the polynomial algebra (so in particular we give a second proof of Theorem~\ref{thm:main}).

\section{Preliminaries on separating invariants} \label{sec:prel} 

Our group $G$ is reductive, therefore the Zariski closure of each $G$-orbit $G\cdot v$ in $V$ contains a unique 
Zariski closed $G$-orbit $G\cdot w$. 
Obviously for any $f \in \mathcal{O}(V)^G$ we have $f(v)=f(w)$. 
Moreover, a subset $S\subseteq \mathcal{O}(V)^G$ is separating if and  only if for any pair $w_1,w_2\in V$ having  distinct Zariski closed $G$-orbits there exists an $f\in S$ with $f(w_1)\neq f(w_2)$ (see for example \cite[Theorem 2.3.6]{derksen-kemper}). 

For a subset $J\subseteq \{1,\dots,n\}$ denote by $V_J$ the subspace $\bigoplus_{j\in J}V_j$ of $V$. 

\begin{lemma}\label{lemma:2} 
If $S\subseteq \mathcal{O}(V)^G$ is a separating set in $\mathcal{O}(V)^G$, then 
$\{f\vert_{V_J}\mid f\in S\}$ is a separating set in $\mathcal{O}(V_J)^G$ for any subset $J\subseteq \{1,\dots,n\}$. In particular, if 
$\{x^m\mid m\in M\}$ is  a separating set in $\mathcal{O}(V)^G$, then 
$\{x^m\mid m\in M_J\}$ is a separating set in $\mathcal{O}(V_J)^G$. 
\end{lemma} 
 
 \begin{proof} 
 Distinct Zariski closed $G$-orbits in $V_J$ are distinct Zariski closed $G$-orbits in $V$, hence they can be separated by an element of $S$. Therefore $S$ restricts to a separating set on $V_J$. Moreover, if $\mathrm{supp}(m)$ is not contained $J$, then $x^m$ restricts to the constant zero function on $V_J$, hence can be omitted from any separating set 
 on $V_J$. 
 \end{proof} 

\begin{lemma}\label{lemma:0} 
For any subset $J\subseteq \{1,\dots,n\}$ and any additive submonoid $Q$ of $\mathbb{N}_0^n$ there exists an $m\in Q$ such that  
\[\mathrm{supp}(m)=\bigcup_{q\in Q\colon \mathrm{supp}(q)\subseteq J}\mathrm{supp}(q).\] 
\end{lemma}

\begin{proof} 
Since $J$ is finite, there exist finitely many $q_1,\dots,q_l\in Q$ such that 
\[\bigcup_{q\in Q\colon \mathrm{supp}(q)\subseteq J}\mathrm{supp}(q)=
\bigcup_{i=1}^l\mathrm{supp}(q_i).\]  
On the other hand, $\bigcup_{i=1}^l\mathrm{supp}(q_i)=\mathrm{supp}(\sum_{i=1}^lq_i)
=\mathrm{supp}(m)$, where  $m:=\sum_{i=1}^lq_i\in Q$. 
\end{proof} 

By the decomposition \eqref{eq:decomposition of V} any $v\in V$ is uniquely written as 
$v=\sum_{i=1}^nv_i$, where $v_i\in V_i$. Set $\mathrm{supp}(v):=\{i\in\{1,\dots,n\}\mid v_i\neq 0\}$. 
For a subset $M\subseteq\mathbb{Z}^n$ denote by $\mathbb{N}_0 M$ the additive submonoid of $\mathbb{Z}^n$ generated by $M$. 

\begin{remark}\label{remark:Sigma M} 
Clearly, $\{x^m\mid m\in M\}$ is a separating set in $\mathcal{O}(V)^G$ if and only if 
$\{x^m\mid m\in \mathbb{N}_0 M\}$ is a separating set in  $\mathcal{O}(V)^G$. 
\end{remark} 

\begin{lemma}\label{lemma:1} Let $\{x^m\mid m\in M\}$ be a separating set in $\mathcal{O}(V)^G$ 
for some $M\subseteq \mathcal{B}$. Then the $G$-orbit $G\cdot v$ of $v\in V$ is Zariski closed in $V$ if and only if there exists an 
$m\in \mathbb{N}_0 M$ with $\mathrm{supp}(m)=\mathrm{supp}(v)$.  
\end{lemma} 

\begin{proof} 
Applying Lemma~\ref{lemma:0} for $J:=\mathrm{supp}(v)$ and $Q:=\mathbb{N}_0 M$ we conclude the existence of 
$m\in \mathbb{N}_0 M$ with $\mathrm{supp}(m)\subseteq \mathrm{supp}(v)$ and 
$\mathrm{supp}(m)\supseteq \mathrm{supp}(m')$ 
for any $m'\in \mathbb{N}_0 M$ with $\mathrm{supp}(m')\subseteq \mathrm{supp}(v)$.  
Set $w:=\sum_{i\in\mathrm{supp}(m)}v_i$, so $\mathrm{supp}(w)=\mathrm{supp}(m)$ and 
$x_i(v)=x_i(w)$ for all $i\in \mathrm{supp}(m)$. 
We claim that $x^q(v)=x^q(w)$ for all $q\in M$. 
Indeed, if  $\mathrm{supp}(q)\nsubseteq \mathrm{supp}(v)$ then $x^q(v)=0=x^q(w)$, whereas if 
$\mathrm{supp}(q)\subseteq \mathrm{supp}(v)$, then 
$\mathrm{supp}(q)\subseteq \mathrm{supp}(m)$ by choice of $m$, and so  
\[x^q(w)=\prod_{i\in\mathrm{supp}(q)}x_i(w)^{q_i}=\prod_{i\in\mathrm{supp}(q)}x_i(v)^{q_i}=x^q(v).\]  

Assume now that the orbit $G\cdot v$ is Zariski closed. Since $x^q(w)=x^q(v)$ for all $q\in M$, we have 
$f(w)=f(v)$ for all $f\in \mathcal{O}(V)^G$, implying that $v$ lies in the Zariski closure of $G\cdot w$. 
If $\mathrm{supp}(w)\subsetneq \mathrm{supp}(v)$, then taking $i\in \mathrm{supp}(v)\setminus \mathrm{supp}(w)$, 
we have that $G\cdot w$ is contained in the $G$-stable hyperplane $L_i:=\bigoplus_{j\in \{1,\dots,n\}\setminus \{i\}}V_j$ 
in $V$, whereas $v\notin L_i$. Since $L_i$ is both $G$-stable and Zariski closed, this is a contradiction. 
It follows that $\mathrm{supp}(m)=\mathrm{supp}(w)=\mathrm{supp}(v)$. 

Conversely, assume $\mathrm{supp}(m)=\mathrm{supp}(v)$. Since $v$ is contained in the $G$-stable Zariski closed 
subset $\bigoplus_{i\in \mathrm{supp}(v)}V_i$, we conclude that the Zariski closure $\overline{G\cdot v}$ is also contained in $\bigoplus_{i\in \mathrm{supp}(v)}V_i$; that is, $\mathrm{supp}(u)\subseteq \mathrm{supp}(v)$ for any 
$u\in \overline{G\cdot v}$. Suppose for contradiction that $G\cdot v$ is not Zariski closed. Then 
$G^\circ\cdot v$ is not Zariski closed, hence there exists an element 
$u\in \overline{G^\circ\cdot v}\setminus G^\circ\cdot v$.  Then 
$\dim(\mathrm{Stab}_{G^\circ}(u))>\dim(\mathrm{Stab}_{G^\circ}(v))$. 
On the other hand, for any $z\in V$ we have $\dim(\mathrm{Stab}_{G^\circ}(z))=\bigcap_{j\in \mathrm{supp}(z)}\ker(\chi_j\vert_{G^\circ})$. It follows that $\mathrm{supp}(u)\subsetneq \mathrm{supp}(v)$, 
implying in turn that $x^m(u)=0\neq x^m(v)$. This contradicts to the assumption $u\in \overline{G\cdot v}$. 
Consequently, $G\cdot v$ is Zariski closed. 
\end{proof} 


\section{Proof of Theorem~\ref{thm:main}} \label{sec:1=2}

\begin{lemma}\label{lemma:3} 
Supose that the characteristic of $K$ is zero, and $M\subseteq \mathcal{B}$ such that $\{x^m\mid m\in M\}$ is a separating system in $\mathcal{O}(V)^G$. Then the abelian subgroup $\mathbb{Z}\mathcal{B}$ of $\mathbb{Z}^n$ is generated by $M$.  
\end{lemma}

\begin{proof} 
Set $J:=\bigcup_{q\in\mathcal{B}}\mathrm{supp}(q)$. Then we have $\mathcal{B}=\mathcal{B}_J$ and hence $M=M_J$. Moreover, the monomials $\{x^m\mid m\in M\}$ form a separating set in $\mathcal{O}(V_J)^G$ 
by Lemma~\ref{lemma:2}. 
Thus to prove our statement, we may restrict to the case when $J=\{1,\dots,n\}$. 
Then there exists an $m\in \mathcal{B}$ with $\mathrm{supp}(m)=\{1,\dots,n\}$ 
by Lemma~\ref{lemma:0}, and consequently, 
the $G$-orbit $G\cdot v$ is Zariski closed for any $v\in V$ with $\mathrm{supp}(v)=\{1,\dots,n\}$ 
by Lemma~\ref{lemma:1}. 

Denote by $\bar G$ the factor of $G$ modulo $\cap_{i=1}^n\ker(\chi_i)$. Then the action of $G$ on $V$ induces a 
faithful action of $\bar G$ on $V$ with $\mathcal{O}(V)^G=\mathcal{O}(V)^{\bar G}$, and we have $\mathcal{B}=\mathcal{B}(\bar\chi_1,\dots,\bar\chi_n)$, where we denote by $\bar\chi_i\in X(\bar G)$ the character of $\bar G$ induced by $\chi_i$. Therefore passing from $G$ to $\bar G$, we may assume at the outset that $G$ acts faithfully on $V$. This means that the character group $X(G)$ is generated by $\chi_1,\dots,\chi_n$. 
Denoting by $e_1,\dots,e_n$ the standard generators of $\mathbb{Z}^n$, the homomorphism $\tilde\pi:e_i\mapsto \chi_i$ ($i=1,\dots,n$)  
factors through a surjective abelian group homomorphism $\pi:\mathbb{Z}^n/\mathbb{Z}M\to X(G)$, 
$e_i+\mathbb{Z}M\mapsto \chi_i$ (because $\ker(\tilde\pi)\supseteq\mathcal{B}\supseteq M$). 

Since $K$ has characteristic zero, the group algebra $K[\mathbb{Z}^n/\mathbb{Z}M]$ is a  commutative reduced 
Hopf $K$-algebra (it contains no non-zero nilpotent elements, see e.g. 
\cite[Theorem 3.1]{wallace}), therefore it is the coordinate ring $\mathcal{O}(H)$ 
of a diagonalizable group $H$ with character group 
$X(H)=\mathbb{Z}^n/\mathbb{Z}M$ (see \cite[8.3 Remark]{borel}). 
Set $\psi_i:=e_i+\mathbb{Z}M$ ($i=1,\dots,n$), where $e_1,\dots,e_n$ are the standard generators of $\mathbb{Z}^n$, so $\pi(\psi_i)=\chi_i$. 
The surjection $\pi:X(H)\twoheadrightarrow X(G)$ extends to a surjective $K$-algebra homomorphism 
between the group algebras $\mathcal{O}(H)=K[X(H)]\twoheadrightarrow K[X(G)]=\mathcal{O}(G)$. 
We keep the notation $\pi$ for this map from $\mathcal{O}(H)$ to $\mathcal{O}(G)$. 
The comorphism $\pi^*$ gives an embedding (injective homomorphism of algebraic groups)  
$\pi^*:G\hookrightarrow H$. 

We claim that $\pi^*$ is surjective onto $H$ (and hence is an isomorphism between $G$ and $H$). To prove this claim consider the representation of $H$ on $V$ given by 
\[h\cdot v:=\psi_1(h)v_1+\cdots+\psi_n(h)v_n, \quad v_i\in V_i, \ v=v_1+\cdots+v_n.\]  
Fix now an arbitrary $v\in V$ with $\mathrm{supp}(v)=\{1,\dots,n\}$. 
For any $m\in M$ we have 
\begin{equation}
\label{eq:x^m(v)=x^m(hv)}x^m(h\cdot v)=\prod_{i=1}^nx_i^{m_i}(\psi_i(h)v_i)=
\prod_{i=1}^m\psi_i(h)^{m_i}x_i^{m_i}(v_i)=x^m(v),\end{equation}
since by construction of $H$ we have $\prod_{i=1}^n\psi_i^{m_i}=\mathbf{1}_{X(H)}$. 
Note that $v$ and $h\cdot v$ both have support $\{1,\dots,n\}$. As we pointed out in the first paragraph  of this proof, $v$ and $h\cdot v$ both have Zariski closed $G$-orbits. 
On the other hand, $x^m(v)=x^m(h\cdot v)$ for all $x\in M$ by \eqref{eq:x^m(v)=x^m(hv)}. 
We conclude that $h\cdot v$ lies on the $G$-orbit of $v$; that is, there exists a $g\in G$ (depending on $h$) such that 
$g\cdot v=h\cdot v$. It follows that $\chi_i(g)=\psi_i(h)$ for all $i=1,\dots,n$. 
Consequently, 
$\psi_i(h)=\chi_i(g)=(\pi(\psi_i))(g)=\psi_i(\pi^*(g))$ for all $i=1,\dots,n$. Since the coordinate ring of $H$ is generated by 
$\psi_1,\dots,\psi_n$, we conclude that $h=\pi^*(g)$.  

Thus we showed that $\pi^*$ is an isomorphism of algebraic groups, impying in turn that $\pi:X(H)=\mathbb{Z}^n/\mathbb{Z}M\to X(G)$ is an isomorphism of abelian groups. 
That is, $\mathbb{Z}M$ generates the abelian group 
$\mathcal{G}:=\{q\in \mathbb{Z}^n\mid \prod_{i=1}^n\chi_i^{q_i}=\mathbf{1}\in X(G)\}$. 
Note finally the equality $\mathcal{G}=\mathbb{Z}\mathcal{B}$. Indeed, 
we have just proved $\mathcal{G}\subseteq\mathbb{Z}M\subseteq\mathbb{Z}\mathcal{B}$, 
whereas the obvious inclusion 
$\mathcal{G}\supseteq \mathcal{B}$ implies $\mathcal{G}\supseteq \mathbb{Z}\mathcal{B}$. 
\end{proof} 

\begin{lemma} \label{lemma:4} 
Let $M$ be a subset of $\mathcal{B}(\chi_1,\dots,\chi_n)$  such that 
for any subset $J\subseteq \{1,\dots,n\}$, the abelian group $\mathbb{Z}\mathcal{B}_J$ is 
generated by $M_J$. Then 
the monomials $\{x^m\mid m\in M\}$ form a separating set in $\mathcal{O}(V)^G$.
\end{lemma}

\begin{proof}  
Take any $b\in \mathcal{B}$, $v,w\in V$ such that $x^b(v)\neq x^b(w)$. 
What we have to show is that $x^m(v)\neq x^m(w)$ for some $m\in \mathbb{N}_0M$ 
(see Remark~\ref{remark:Sigma M}). Set $J:=\mathrm{supp}(b)$. If none of $\mathrm{supp}(v)$ and $\mathrm{supp}(w)$ contains $J$, then $x^b(v)=0=x^b(w)$ is a contradiction. Hence say $\mathrm{supp}(v)\supseteq J$. By assumption there exist $m,q\in \mathbb{N}_0M_J$ with $b=m-q$. Necessarily we have $\mathrm{supp}(q)\subseteq \mathrm{supp}(m)=J$: indeed, $m-q\in\mathbb{N}_0^n$ implies $\mathrm{supp}(q)\subseteq \mathrm{supp}(m)$, and $b=m-q$ implies $J=\mathrm{supp}(b)\subseteq\mathrm{supp}(m)\subseteq J$ (the last inclusion holds by $m\in \mathbb{N}_0M_J$). 
 If $\mathrm{supp}(w)$ does not contain $J$, then $x^m(w)=0$ and $x^m(v)\neq 0$, and we are done. Otherwise both of $\mathrm{supp}(v)$ and $\mathrm{supp}(w)$ contain $J$, hence $x^b(v)=\frac{x^m(v)}{x^q(v)}$ and 
$x^b(w)=\frac{x^m(w)}{x^q(w)}$. If follows by $x^b(v)\neq x^b(w)$ that 
$x^m(v)\neq x^m(w)$ or $x^q(v)\neq x^q(w)$. 
\end{proof} 

\begin{proofof}{Theorem~\ref{thm:main}}
The implication (2) $\Longrightarrow$  (1) is the content of Lemma~\ref{lemma:4}. 
The implication (1) $\Longrightarrow$ (2) is a consequence of Lemma~\ref{lemma:3} and Lemma~\ref{lemma:2}. \end{proofof}
\goodbreak


\section{Monoid theoretic characterization of $\mathcal{B}$} \label{sec:monoid} 

Before turning to the material needed for the proof of Theorem~\ref{thm:ZB_J}, 
we want to clarify its monoid theoretic content independent of diagonalizable group actions. 

\begin{definition}\label{def:difference-closed} 
We call a submonoid $B$ of the additive monoid $\mathbb{N}_0^n$  \emph{difference-closed} if it satisfies any (hence all) of the following equivalent conditions:
\begin{itemize}
\item[(i)] For any $m,q\in B$ with $m-q\in \mathbb{N}_0^n$ we have  $m-q\in B$. 
\item[(ii)] We have $B=\mathbb{Z}B\cap \mathbb{N}_0^n$. 
\item[(iii)] We have $B=H\cap \mathbb{N}_0^n$ for some subgroup of the additive group $\mathbb{Z}^n$. 
\end{itemize} 
\end{definition}  

\begin{remark} 
Condition (i) in 
Definition~\ref{def:difference-closed} is expressed 
in \cite[Chapter 2.4]{geroldinger_halter-koch} by saying that the embedding $B\hookrightarrow \mathbb{N}_0^n$ is a \emph{divisor homomorphism}. 
Such monoids are finitely generated Krull monoids. The notion of Krull monoids plays a central role in the theory of non-unique factorizations, see the book \cite{geroldinger_halter-koch}. Difference-closed submonoids of $\mathbb{N}_0^n$ appear in the literature under various names. They are called \emph{pure submonoids} of 
$\mathbb{N}_0^n$ in \cite[Section 2.B, p. 64]{bruns-gubeladze}, where it is shown in  \cite[Theorem 2.29]{bruns-gubeladze} that a positive affine monoid is normal if and only if it can be realized as a pure submonoid of $\mathbb{N}_0^n$. 
\end{remark} 

\begin{proposition}\label{prop:difference-closed} 
The following conditions are equivalent for a submonoid $B$ of the additive monoid 
$\mathbb{N}_0^n$: 
\begin{itemize} 
\item[(i)] $B$ is difference-closed. 
\item[(ii)] For any algebraically closed field $K$ whose characteristic does not divide the order of the torsion subgroup $A$ of $X:=\mathbb{Z}^n/\mathbb{Z}B$, there exists a diagonalizable group $G$ over $K$ and characters $\chi_1,\dots,\chi_n$ such that 
$B=\mathcal{B}(\chi_1,\dots,\chi_n)$ and $X\cong X(G)$ 
(so in particular, $\dim(G)=\mathrm{rk}(X/A)$). 
\end{itemize}
\end{proposition} 

\begin{proof}  This is the content \cite[Theorem 5.17, Corollary 5.18]{bruns-gubeladze}.  
We include a proof using the terminology and notation of the present paper. 
A monoid of the form $\mathcal{B}(\chi_1,\dots,\chi_n)$ is by definition the intersection of 
a subgroup of $\mathbb{Z}^n$ and $\mathbb{N}_0^n$, hence is difference-closed. 
So (ii) implies (i). 

Let now $B$ be a difference-closed submonoid of $\mathbb{N}_0^n$. 
Consider the finitely generated abelian group $X:=\mathbb{Z}^n/\mathbb{Z}B$, and denote by $A$ its torsion subgroup. Let $K$ be an algebraically closed field whose characteristic does not divide $|A|$. Then the group algebra $K[X]$ is a finitely generated commutative reduced  Hopf algebra (see e.g. \cite[Theorem 3.1]{wallace}), so it can be identified with the coordinate ring $\mathcal{O}(G)$ of a diagonalizable group $G$ over $K$ (see \cite[8.3 Remark]{borel}). Then $X$ is identified with the character group $X(G)$. 
Set $\chi_i:=e_i+\mathbb{Z}B\in X$ for $i=1,\dots,n$, where $e_1,\dots,e_n$ are the standard generators of $\mathbb{Z}^n$. 
We have that 
\[\mathcal{B}(\chi_1,\dots,\chi_n)=\mathbb{Z}B\cap\mathbb{N}_0^n=B\] 
(the first equality holds by construction of $\chi_1,\dots,\chi_n$, and the second equality holds because $B$ is difference-closed). 
Note finally that the diagonalizable group decomposes as $G^\circ\times \hat A$, where $\hat A$ is a finite abelian group whose order is not divisible by $\mathrm{char}(K)$ 
(see \cite[8.7 Proposition]{borel}). Therefore $X(G)=X(G^{\circ})\times X(\hat A)$. 
Here $X(\hat A)\cong A$ is the torsion subgroup of $X(G)$, and 
$X(G)/X(\hat A)\cong X(G^\circ)$ is a free abelian group of rank  $\dim(G)$. 
This finishes the proof of the implication (i)$\Longrightarrow$(ii). 
\end{proof} 
 
By Proposition~\ref{prop:difference-closed}, our Theorem~\ref{thm:ZB_J} 
and its sharpening in Remark~\ref{remark:tau(B) for tori} for tori (based on 
\cite[Theorem 6.7]{dufresne-jeffries:2}) have the following reformulation: 

\begin{corollary}\label{cor:reformulation} 
Let $B$ be a difference-closed submonoid in $\mathbb{N}_0^n$. 
Then 
\[B\subseteq \mathbb{Z}\{m\in B\colon |\mathrm{supp}(m)|\le 1+3s+\mathrm{rk}(A)\},\] 
where $A$ is the torsion subgroup of $X:=\mathbb{Z}^n/\mathbb{Z}B$ and 
$s=\mathrm{rk}(X/A)$. 
Moreover, when $A$ is trivial, then 
\[B\subseteq \mathbb{Z}\{m\in B\colon |\mathrm{supp}(m)|\le 1+2s\}.\] 
\end{corollary}


\section{Helly dimension of diagonalizable groups} \label{sec:helly} 

The \emph{Helly dimension} $\kappa(G)$ of $G$ is the minimal positive integer $d$ such that any finite system of Zariski closed cosets in $G$ having empty intersection has a subsystem of at most $d$ cosets with empty intersection 
(for the trivial group define the Helly dimension to be $1$).  
This quantity was introduced for finite groups in \cite{domokos_typical}, and for algebraic groups in \cite{domokos-szabo}.  
It was proved in \cite{domokos-szabo} that the Helly dimension of a linear algebraic group over a field of charcteristic zero is finite. For diagonalizable groups it is finite also in positive characteristic as well, and we give an upper bound for it below. 
 
\begin{proposition} \label{prop:diagonalizable Helly} 
For a diagonalizable group  $G$ we have 
\[1+\mathrm{rk}(X(G))\le \kappa(G)\le1+\dim(G^\circ)+\mathrm{rk}(X(G)).\] 
\end{proposition} 

\begin{proof} 
Apply induction on $\dim(G)=\dim(G^\circ)$. If $\dim(G^\circ)=0$, then $G$ is finite, $G\cong X(G)$, hence 
$\mathrm{rk}(X(G))=\mathrm{rk}(G)$. Moreover, by \cite[Corollary 2.3]{domokos-szabo} we have $\kappa(G)=1+\mathrm{rk}(G)$. So the statement holds when $\dim(G^\circ)=0$. 

Suppose next that $\dim(G^\circ)>0$, and the inequality 
$\kappa(G)\le1+\dim(G^\circ)+\mathrm{rk}(X(G))$  
holds for diagonalizable groups of smaller dimension. 
Set $d:=1+\dim(G^\circ)+\mathrm{rk}(X(G))$, and take Zariski closed cosets 
$g_1H_1,\dots,g_tH_t$ such that any $d$ of these cosets have a common element. We need to show that the intersection of all of these cosets is non-empty. 
Assume first that $\dim(H_i)=\dim(G)$ for all $i=1,\dots,t$. That is, $H_i^\circ=G^\circ$ for all $i$. Consider the natural surjection $\eta:G\twoheadrightarrow G/G^\circ$. Then $\eta(g_iH_i)$, $i=1,\dots,t$ are cosets in the finite abelian group 
$G/G^\circ$ such that any $d$ of them have non-empty intersection. 
As $d>1+\mathrm{rk}(X(G))>1+\mathrm{rk}(G/G^{\circ})=\kappa(G/G^\circ)$  
by \cite[Corollary 2.3]{domokos-szabo}, we conclude that 
the intersection of the cosets $\eta(g_iH_i)$ is non-empty. It follows that 
\[\emptyset\neq \eta^{-1}(\bigcap_{i=1}^t\eta(g_iH_i))\subseteq \bigcap_{i=1}^t\eta^{-1}(\eta(g_iH_i))
=\bigcap_{i=1}^tg_iH_i,\] 
and we are done in this case. 

Otherwise there is some $i$ with $\dim(H_i)<\dim(G)$. Without loss of generality we may assume 
that $\dim(H_1)<\dim(G)$. Note that $d\ge 2$, hence $g_1H_1\cap g_iH_i$ is non-empty for each $i=2,\dots,t$ by assumption, 
implying that $C_i:=H_1\cap g_1^{-1}g_iH_i$, $i=2,\dots,m$ are Zariski closed cosets in the group $H_1$. 
For any $2\le i_1<\cdots<i_{d-1}\le t$ we have 
\[C_{i_1}\cap\dots\cap C_{i_{d-1}}=g_1^{-1}(g_1H_1\cap g_{i_1}H_{i_1}\cap\dots\cap g_{i_{d-1}}H_{i_{d-1}})\neq \emptyset.\] 
Note that $X(H_1)$ is a homomorphic image of $X(G)$, hence 
$\mathrm{rk}(X(H_1))\le\mathrm{rk}(X(G))$, and so 
\[d-1=1+(\dim(G^\circ)-1)+\mathrm{rk}(X(G))\ge 1+\dim(H_1^\circ)+\mathrm{rk}(X(H_1))\ge \kappa(H_1)\] 
(the last inequality holds by the induction hypothesis).  
Therefore by definition of $\kappa(H_1)$ we have $C_2\cap\dots\cap C_t\neq\emptyset$, implying in turn that 
\[\emptyset\neq g_1(C_2\cap\dots\cap C_t)=g_1H_1\cap\dots \cap g_tH_t.\] 
Thus the inequality $\kappa(G)\le1+\dim(G^\circ)+\mathrm{rk}(X(G))$ is proved. 

In order to prove the inequality $\kappa(G)\ge 1+\mathrm{rk}(X(G))$ observe that 
$G=G^\circ\times A$, and $G^\circ$ has a finite subgroup $H$ such that 
$\mathrm{rk}(H)=\dim(G^\circ)$ and $\mathrm{rk}(H\times A)=\mathrm{rk}(H)+\mathrm{rk}(A)$. Then we have $\kappa(G)\ge\kappa(H\times A)$, and by \cite[Corollary 2.3]{domokos-szabo} we have 
$\kappa(H\times A)=1+\mathrm{rk}(H\times A)=1+\dim(G^\circ)+\mathrm{rk}(G/G^\circ)=
1+\mathrm{rk}(X(G))$. 
The proof is finished. 
\end{proof} 

\begin{remark}\label{remark:Helly K^times} 
For a finite group $G$ 
the upper bound for $\kappa(G)$ in Proposition~\ref{prop:diagonalizable Helly} is sharp. 
On the other hand for $G=K^\times$ Proposition~\ref{prop:diagonalizable Helly} gives 
$\kappa(K^\times)\le 3$. However, we have $\kappa(K^\times)=2$, so the upper bound for 
$\kappa(G)$ in Proposition~\ref{prop:diagonalizable Helly} is not sharp in general when $\dim(G)>0$. 
Indeed, assume $t\ge 3$, and $g_1H_1,\dots,g_tH_t$ are cosets in $K^\times$ such that any two have non-empty intersection. We shall show that $\cap_{i=1}^tg_iH_i\neq\emptyset$. 
It is sufficient to deal with the case when all the $H_i$ are proper subgroups of $G$. 
Then all the $H_i$ are finite. Moreover, multiplying the cosets from the left by $g_1^{-1}$ we may reduce to the case when $g_1H_1=H_1$. Since $H_1\cap g_iH_i$ is non-empty, we conclude that $g_iH_i$ is contained in the subgroup $\langle H_1,H_i\rangle$. So all our cosets are contained in the subgroup $H:=\langle H_1,\dots,H_t\rangle$ of $K^\times$. 
Now $H$ is finite, hence cyclic. Therefore $\kappa(H)\le 2$ by \cite[Proposition 4.3]{domokos} or \cite[Corollary 2.3]{domokos-szabo}, implying in turn that our $t$ cosets have non-empty intersection. 
\end{remark}

 \section{Proof of Theorem~\ref{thm:ZB_J}} \label{sec:proof of thm ZB_J} 
 
 For $v\in V$ and $J\subseteq \{1,\dots,n\}$ set $v_J:=\sum_{i\in J}v_i$. 
 We set $v_{\emptyset}=0\in V$. 
 Following \cite[Definition 5.1]{domokos-szabo}, we define $\delta(G,V)$ as the minimal non-negative integer $\delta$ such that 
 for any $v\in V$ with Zariski closed $G$-orbit, there exists a subset $J\subseteq \{1,\dots,n\}$ with $|J|=\delta$ such that 
 $G\cdot v_J$ is Zariski closed and $\dim(G\cdot v)=\dim(G\cdot v_J)$. 
 Note that by \cite[Proposition 5.2]{domokos-szabo}, if $G\cdot v_J$ is Zariski closed and $\dim(G\cdot v)=\dim(G\cdot v_J)$, then for any $L\supseteq J$ we have that $G\cdot v_L$ is Zariski closed and $\dim(G\cdot v)=\dim(G\cdot v_L)$. 
 
\begin{proposition}\label{prop:delta} 
We have the inequality $\delta(G,V)\le 2\dim(G)$. 
\end{proposition}  

\begin{proof} The $G$-orbit of an element $w$ of $V$ or $V_J$  (where $J\subseteq \{1,\dots,n\}$) 
is Zariski closed if and only if its $G^\circ$-orbit is Zariski closed. Moreover, $\dim(G\cdot w)=\dim(G^\circ\cdot w)$. 
Consequently, $\delta(G,V)=\delta(G^\circ,V)$. The group $G^\circ$ is a torus, and  the  inequality 
$\delta(G^\circ,V)\le 2\dim(G^\circ)$  
is  proved in \cite[Proposition 5.5]{domokos-szabo} as a corollary of a .Caratheodory  type theorem in convex geometry. 
\end{proof} 

\begin{definition}\label{def:tau(G,V)}
Let $\tau(G,V)$ denote the minimal non-negative integer $t$ such that 
$\bigcup_{J\subseteq \{1,\dots,n\}\colon |J|\le t}\mathcal{O}(V_J)^G$ 
is a separating set in $\mathcal{O}(V)^G$ (cf.  \cite[Definition 5.8]{domokos-szabo}). 
\end{definition} 

A straightforward rewording of the proof of \cite[Lemma 5.9]{domokos-szabo} 
implies the inequality $\tau(G,V)\le \delta(G,V)+\kappa(G)$, which by 
Proposition~\ref{prop:diagonalizable Helly} yields 
$\tau(G,V)\le 1+\delta(G,V)+\dim(G^\circ)+\mathrm{rk}(X(G))$. 
Here we prove a stronger upper bound for $\tau(G,V)$ of similar nature. 

 \begin{lemma} \label{lemma:tau} 
 We have the inequality 
 \[\tau(G,V)\le 1+\delta(G,V)+\mathrm{rk}(X(G)).\] 
 \end{lemma} 
 
 \begin{proof} 
 Set $t:=1+\delta(G,V)+\mathrm{rk}(X(G))$. Assume that the orbits $G\cdot v$ and 
 $G\cdot w$ are Zariski closed, and $f(v)=f(w)$ for all $f\in \bigcup_{|J|\le t}\mathcal{O}(V_J)^G$. Then for any $J\subseteq \{1,\dots,n\}$ with $|J|\le t$, the Zariski closures of $G\cdot v$ and $G\cdot w$ have non-empty intersection. We have to show that $G\cdot v=G\cdot w$. 
 
By symmetry we may assume $\dim(G\cdot v)\ge \dim(G\cdot w)$. By definition of $\delta(G,V)$ there exists a subset $I\subseteq \mathrm{supp}(v)$ with 
$|I|\le \delta(G,V)$ such that $G\cdot v_I$ is Zariski closed and 
$\dim(G\cdot v_I)=\dim(G\cdot v)$. By assumption the Zariski closures of 
$G\cdot v_I$ and $G\cdot w_I$ have non-empty intersection, and 
$\dim(G\cdot v_I)=\dim(G\cdot v)\ge \dim(G\cdot w)\ge\dim(G\cdot w_I)$. 
It follows that $G\cdot v_I=G\cdot w_I$. Thus replacing $w$ by an appropriate element in its orbit we may assume that $v_I=w_I$. By \cite[Proposition 5.2]{domokos-szabo} for any $J\supseteq I$ we have that the $G$-orbits of $v_J$ and $w_J$ are Zariski closed. 
In particular, for all $j\in \{1,\dots,n\}\setminus I$, the $G$-orbit of $v_{I\cup\{j\}}$ is Zariski closed, implying by Lemma~\ref{lemma:1} that there exists an $m\in \mathcal{B}_{I\cup\{j\}}$ with $\mathrm{supp}(m)=I\cup\{j\}$. By assumption we have $x^m(v)=x^m(w)$. It follows that 
\[j\in \mathrm{supp}(v)\iff x^m(v)\neq 0 \iff x^m(w)\neq 0\iff j\in \mathrm{supp}(w).\]
We infer that $\mathrm{supp}(v)=\mathrm{supp}(w)$. Clearly we may assume that 
$\mathrm{supp}(v)=\{1,\dots,n\}$. Then the stabilizer of $v$ is $\cap_{i=1}^n\ker(\chi_i)$, the kernel of the action of $G$ on $V$. Consider the induced representation of 
$\bar G=G/\cap_{i=1}^n\chi_i$ on $V$. The stabilizer of $v$ in $\bar G$ is trivial, hence 
$\dim(G\cdot v)=\dim(\bar G\cdot v)=\dim(\bar G)$. Then we have 
\[\dim(\bar G\cdot v_I)=\dim(G\cdot v_I)=\dim(G\cdot v)=\dim(\bar G\cdot v)=\dim(\bar G),\] 
implying that the stabilizer $H$ of $v_I$ in $\bar G$ is finite. 

For any subset $L\subseteq \{1,\dots,n\}\setminus \{I\}$ with $|L|\le 1+\mathrm{rk}(X(G))$, 
the orbits $G\cdot v_{I\cup L}$ and $G\cdot w_{I+L}$ are Zariski closed on one hand, and intersect nontrivially by assumption on the other hand, thus 
$\bar G\cdot v_{I\cup L}=\bar G\cdot w_{I\cup L}$. Since $v_I=w_I$, this means that there exists an $h\in H$ with $h\cdot v_L=w_L$. 
Then $h\in \cap_{j\in L}C_j$, where 
for $j\in \{1,\dots,n\}\setminus I$,  $C_j:=\{g\in H\mid g\cdot v_j=w_j\}$ is a coset in $H$. 
So any $1+\mathrm{rk}(X(G))$ of these cosets have non-empty intersection. As $H$ is a subgroup of a factor group of $G$, we have that $X(H)$ is   a factor group of a subgroup of $X(G)$, and thus $\mathrm{rk}(H)=\mathrm{rk}(X(H))\leq\mathrm{rk}(X(G))$, implying in turn by 
\cite[Corollary 2.3]{domokos-szabo} that 
$\kappa(H)\le 1+\mathrm{rk}(X(G))$.  
By definition of $\kappa(H)$ this means that $\cap_{j\in \{1,\dots,n\}\setminus I} C_j$ is non-empty. Take an element $h$ from this intersection, then we get $h\cdot v=w$. 
 \end{proof} 

\begin{remark} \label{remark:dufresne-jeffries}
In the special case when $G=G^\circ$ is a torus, Proposition~\ref{prop:delta} and 
Lemma~\ref{lemma:tau} yield the inequality $\tau(G,V)\le 1+3\dim(G)$. 
However, for this case the stronger result $\tau(G,V)\le 1+2\dim(G)$ is proved by 
Dufresne and Jeffries in \cite[Theorem 6.7]{dufresne-jeffries:2}. 
\end{remark} 
  
We record a consequence of Theorem~\ref{thm:main}: 

  \begin{corollary}\label{cor:tau=tau} 
  Assume that $\mathrm{char}(K)=0$. Then we have the equality $\tau(G,V)=\tau(\mathcal{B})$. 
  \end{corollary}
  
  \begin{proof} 
 Set  $M:=\bigcup_{J\subseteq \{1,\dots,n\}\colon |J|\le \tau(\mathcal{B})} \mathcal{A}_J$. Then by Definition~\ref{def:tau(B)} 
 for all $I\subseteq \{1,\dots,n\}$ we have that $\mathcal{A}_I\subseteq \mathbb{Z}M_I$, 
 therefore by the implication (2) $\Longrightarrow$ (1) of Theorem~\ref{thm:main} we infer 
 that $\{x^m\mid m\in M\}$ is a separating set in $\mathcal{O}(V)^G$. 
 As $\{x^m\mid m\in M\}\subseteq \bigcup_{J\subseteq \{1,\dots,n\}\colon |J|\le \tau(\mathcal{B})}\mathcal{O}(V_J)^G$, the latter is also a separating set in  $\mathcal{O}(V)^G$. 
 This shows the inequality $\tau(G,V)\le \tau(\mathcal{B})$. 
 
 In order to prove the reverse inequality set 
$M:=\bigcup_{J\subseteq \{1,\dots,n\}\colon |J|\le \tau(G,V)} \mathcal{A}_J$. 
 Then  by Definition~\ref{def:tau(G,V)} and by Proposition~\ref{prop:monoid generators} we have that 
 $\{x^m\mid m\in M\}$ is a separating set in $\mathcal{O}(V)^G$. 
 By the implication (1) $\Longrightarrow$ (2) of Theorem~\ref{thm:main} we infer 
 that for any subset $I\subseteq \{1,\dots,n\}$, $\mathcal{A}_I\subseteq \mathbb{Z}M_I$. 
 Note that $M_I=\bigcup_{J\subseteq I\colon |J|\le \tau(G,V)} \mathcal{A}_J$. 
 This shows the inequality $\tau(\mathcal{B})\le\tau(G,V)$. 
  \end{proof} 
  
  \bigskip
 \begin{proofof}{Theorem~\ref{thm:ZB_J}} 
  In the statement of this theorem the characteristic of $K$ is arbitrary. 
 However, by Proposition~\ref{prop:difference-closed} we may reduce to the case when 
 $\mathrm{char}(K)=0$, and then by Corollary~\ref{cor:tau=tau} we may conclude 
 that $\tau(\mathcal{B})=\tau(G,V)$. On the other hand, 
 combining  Proposition~\ref{prop:delta} 
Lemma~\ref{lemma:tau} we obtain that $\tau(G,V)\le 1+2\dim (G)+\mathrm{rk}(X(G))$. 
 \end{proofof} 
  
  
 \section{Some conjectures} \label{sec:conjectures}

A solution in the affirmative of Conjecture~\ref{conjecture:helly} below would give a common generalization of the known sharp bounds for the Helly dimension of a finite group and of the $1$-dimensional 
torus $K^\times$  (cf. Remark~\ref{remark:Helly K^times}):  

\begin{conjecture}\label{conjecture:helly} 
The statement of Proposition~\ref{prop:diagonalizable Helly} can be sharpened to 
\[\kappa(G)\le 1+\mathrm{rk}(X(G)).\] 
\end{conjecture} 

A solution in the affirmative to Conjecture~\ref{conjecture:tau} below would give a common generalization of the known sharp bounds for $\tau(\mathcal{B})$ for finite groups and for tori: 

\begin{conjecture}\label{conjecture:tau}  
The statement of Theorem~\ref{thm:ZB_J} can be sharpened to 
\[\tau(\mathcal{B})\le 1+\dim(G)+\mathrm{rk}(X(G)).\] 
\end{conjecture} 


\section{Positive characteristic} \label{sec:positive characteristic} 

The following result is stated for torus actions in \cite{dufresne-jeffries:2}, and the proof works for not necessarily connected diagonalizable groups as well: 

\begin{proposition}\label{prop:dufresne-jeffries pos char} \cite[Proposition 6.1]{dufresne-jeffries:2} Assume that $\mathrm{char}(K)=p>0$. For some $M\subseteq \mathcal{B}$, the monomials $\{x^m\mid m\in M\}$ form a separating set in $\mathcal{O}(V)^G$ if and only if there is a non-negative integer $\alpha$ such that $p^{\alpha}\mathcal{B}\subseteq \mathbb{N}_0M$. 
\end{proposition} 

\begin{definition}\label{def:tau_p}
For a prime $p$ and a submonoid $B$ of the additive monoid $\mathbb{N}_0^n$ 
define $\tau_p(B)$ as the minimal non-negative integer $t$ such 
that for any $I\subseteq \{1,\dots,n\}$, there exists a non-negative integer $\alpha_I$ for which  
$p^{\alpha_I}B_I$ is contained in the abelian subgroup of $\mathbb{Z}^n$ generated by $\bigcup_{J\subseteq I,\  |J|\le t}B_J$. 
\end{definition}

Now we are in position to state an extension of Proposition~\ref{prop:dufresne-jeffries pos char}: 

\begin{theorem}\label{thm:pos char main} 
Assume that $\mathrm{char}(K)=p>0$. 
Then the  following conditions are equivalent for a subset $M\subseteq \mathcal{B}(\chi_1,\dots,\chi_n)$:  
\begin{enumerate} 
\item The monomials $\{x^m\mid m\in M\}$ form a separating set in $\mathcal{O}(V)^G$. 
\item There is a non-negative integer $\alpha$ such that $p^{\alpha}\mathcal{B}\subseteq \mathbb{N}_0M$. 
\item For any subset $J\subseteq \{1,\dots,n\}$, the factor group $\mathbb{Z}\mathcal{B}_J/\mathbb{Z}M_J$ is a finite $p$-group. 
\item For any subset $J\subseteq \{1,\dots,n\}$ with $|J|\le \tau_p(\mathcal{B})$ there exists a non-negative integer $\alpha_J$ such that 
$p^{\alpha_J}\mathcal{A}_J$ is contained in $\mathbb{Z}M_J$. 
\end{enumerate} 
\end{theorem}

Lemma~\ref{lemma:2}, Lemma~\ref{lemma:1}, Proposition~\ref{prop:diagonalizable Helly}, 
Proposition~\ref{prop:delta}, Lemma~\ref{lemma:tau}  are characteristic free. The statements and proofs of Lemma~\ref{lemma:3} and Lemma~\ref{lemma:4} and Corollary~\ref{cor:tau=tau} have to be modified. 

\begin{lemma} \label{lemma:p4} 
Suppose that $\mathrm{char}(K)=p>0$, and 
let $M$ be a subset of $\mathcal{B}$  such that 
for any $b\in \mathcal{B}$ there is a non-negative integer $\alpha$ such that 
$p^{\alpha}b\in  \mathbb{Z}M_J$, where $J=\mathrm{supp}(b)$.  Then 
the monomials $\{x^m\mid m\in M\}$ form a separating set in $\mathcal{O}(V)^G$.
\end{lemma}

\begin{proof}  Take any $b\in \mathcal{B}$, $v,w\in V$ such that $x^b(v)\neq x^b(w)$. 
What we have to show is that $x^m(v)\neq x^m(w)$ for some $m\in \mathbb{N}_0M$ 
(see Remark~\ref{remark:Sigma M}). Set $J:=\mathrm{supp}(b)$. If none of $\mathrm{supp}(v)$ and $\mathrm{supp}(w)$ contains $J$, then $x^b(v)=0=x^b(w)$ is a contradiction. Hence say $\mathrm{supp}(v)\supseteq J$. By assumption there exist $m,q\in \mathbb{N}_0M_J$ and a non-negative integer $\alpha$ with $p^{\alpha}b=m-q$. Necessarily we have $\mathrm{supp}(q)\subseteq \mathrm{supp}(m)=J$. If $\mathrm{supp}(w)$ does not contain $J$, then $x^m(w)=0$ and $x^m(v)\neq 0$, and we are done. Otherwise both of $\mathrm{supp}(v)$ and $\mathrm{supp}(w)$ contain $J$, hence 
$x^q(v)\neq 0$ and $x^q(w)\neq 0$, and 
\[\frac{x^m(v)}{x^q(v)}=x^{bp^{\alpha}}(v)\neq x^{bp^{\alpha}}(w)=\frac{x^m(w)}{x^q(w)}\] 
(note that  by the assumption on the characteristic $x^b(v)\neq x^b(w)$ implies 
$x^{bp^{\alpha}}(v)\neq x^{bp^{\alpha}}(w)$). 
If follows that 
$x^m(v)\neq x^m(w)$ or $x^q(v)\neq x^q(w)$. 
\end{proof}

\begin{proofof}{Theorem~\ref{thm:pos char main}}  
The equivalence of (1) and  (2) is the content of Proposition~\ref{prop:dufresne-jeffries pos char}. 
Assume next that (2) and hence (1) hold. 
By Lemma~\ref{lemma:2} we get that for any $J\subseteq \{1,\dots,n\}$, 
$\{x^m\mid m\in M_J\}$ is a separating set in $\mathcal{O}(V_J)^G$, from which by 
the implication (1) $\Longrightarrow$ (2) we conclude that $p^{\alpha_J}\mathcal{B}_J\subseteq \mathbb{N}_0M_J$ for some $\alpha_J\in \mathbb{N}_0$. 
Since the monoid $\mathcal{B}_J$ is finitely generated, it follows that 
$\mathbb{Z}\mathcal{B}_J/\mathbb{Z}M_J$ is a finite $p$-group. Thus we showed the implication (2) $\Longrightarrow$ (3). 
The implication (3) $\Longrightarrow$  (1) is a consequence of Lemma~\ref{lemma:p4}. 
The implication (3) $\Longrightarrow$  (4) is trivial. It remains to prove the implication (4) $\Longrightarrow$  (3). 
Assume that (4) holds for $M$.  
Take an arbitrary finite subset $J\subseteq \{1,\dots,n\}$. By Definition~\ref{def:tau_p} there exists a non-negative integer 
$\alpha_J$ such that $p^{\alpha_J}\mathcal{B}_J$ is contained in the subgroup of $\mathbb{Z}^n$ generated by $\bigcup_{I\subseteq J,\ |I|\le \tau_p(\mathcal{B})}\mathcal{A}_I$. By assumption on $M$ for all $I\subseteq J$ with $|I|\le \tau_p(B)$ we have 
$p^{\alpha_I}\mathcal{A}_I\subseteq \mathbb{Z}M_I$. Setting $\alpha:=\alpha_J+\max\{\alpha_I\colon I\subseteq J,\ |I|\le\tau_p(\mathcal{B})\}$ we have $p^{\alpha}\mathcal{B}_J
\subseteq \mathbb{Z}M_J$, so (3) holds. 
\end{proofof}

\begin{remark}\label{remark:tau_p<tau}
Condition (4) in Theorem~\ref{thm:pos char main} can be effectively used thanks to 
Theorem~\ref{thm:ZB_J} and the obvious inequality 
\[\tau_p(\mathcal{B})\le\tau(\mathcal{B}).\]  
\end{remark} 

For completeness of the picture we record 
the following corollary of Theorem~\ref{thm:pos char main}: 

\begin{corollary} \label{cor:tau_p=tau} 
Assume that $\mathrm{char}(K)=p>0$. Then we have 
$\tau(G,V)=\tau_p(\mathcal{B})$. 
\end{corollary} 

\begin{proof} 
Set  $M:=\bigcup_{J\subseteq \{1,\dots,n\}\colon |J|\le \tau_p(\mathcal{B})} \mathcal{A}_J$. Then by Definition~\ref{def:tau_p} 
 for all $I\subseteq \{1,\dots,n\}$ we have that $p^{\alpha_I}\mathcal{B}_I\subseteq \mathbb{Z}M_I$ for some non-negative integer $\alpha_I$, 
 therefore by the implication (3) $\Longrightarrow$ (1) of Theorem~\ref{thm:pos char main} we infer 
 that $\{x^m\mid m\in M\}$ is a separating set in $\mathcal{O}(V)^G$. 
 As $\{x^m\mid m\in M\}\subseteq \bigcup_{J\subseteq \{1,\dots,n\}\colon |J|\le \tau_p(\mathcal{B})}\mathcal{O}(V_J)^G$, the latter is also a separating set in  $\mathcal{O}(V)^G$. 
 This shows the inequality $\tau(G,V)\le \tau_p(\mathcal{B})$. 
 
 In order to prove the reverse inequality set 
$M:=\bigcup_{J\subseteq \{1,\dots,n\}\colon |J|\le \tau(G,V)} \mathcal{A}_J$. 
 Then  by Definition~\ref{def:tau(G,V)} and by Proposition~\ref{prop:monoid generators} we have that 
 $\{x^m\mid m\in M\}$ is a separating set in $\mathcal{O}(V)^G$. 
 By the implication (1) $\Longrightarrow$ (3) of Theorem~\ref{thm:pos char main} we infer 
 that for any subset $I\subseteq \{1,\dots,n\}$ there exists a non-negative integer $\alpha_I$ with  $p^{\alpha_I}\mathcal{A}_I\subseteq \mathbb{Z}M_I$. 
 Note that $M_I=\bigcup_{J\subseteq I\colon |J|\le \tau(G,V)} \mathcal{A}_J$. 
 This shows the inequality $\tau_p(\mathcal{B})\le\tau(G,V)$. 
\end{proof}

\section{Degree bounds and examples} \label{sec:examples} 

The polynomial algebra $\mathcal{O}(V)$ is graded in the standard way (the degree $1$ component is spanned by the variables $x_1,\dots,x_n$). The subalgebra $\mathcal{O}(V)^G$ is spanned by homogeneous elements, so it inherits the grading. We shall denote by $\beta(G,V)$ the minimal non-negative integer $d$ such that 
the algebra $\mathcal{O}(V)^G$ is generated by its homogeneous components of degree at most $d$. 
Furthermore, we shall denote by $\beta_{\mathrm{sep}}(G,V)$ the minimal non-negative integer $d$ such that the homogeneus elements of degree at most $d$ constitute a separating set of $\mathcal{O}(V)^G$. 
By the \emph{length} of $m\in \mathbb{N}_0^n$ we mean $|m|=\sum_{i=1}^nm_i$. As an immediate consequence of Proposition~\ref{prop:monoid generators} we get the following: 

\begin{corollary}\label{cor:beta} 
We have the equality  
\[\beta(G,V)=\max\{|m|\mid m\in  \mathcal{A}(\chi_1,\dots,\chi_n)\}.\]  
\end{corollary} 

Our Theorem~\ref{thm:main} and Theorem~\ref{thm:pos char main} yield an analogous characterisation of 
$\beta_{\mathrm{sep}}(G,V)$ in terms of the monoid $\mathcal{B}(\chi_1,\dots,\chi_n)$. 

\begin{corollary}\label{cor:beta sep} 
Assume that $\mathrm{char}(K)=0$. Then we have the equality 
\begin{align*} 
\beta_{\mathrm{sep}}(G,V)&=
\min\{d\in\mathbb{N}_0\mid \forall J\subseteq\{1,\dots,n\}:\ \mathcal{A}_J\subseteq \mathbb{Z}\{m\in\mathcal{A}_J\mid 
|m|\le d\}\}
\\ &=
\min\{d\in \mathbb{N}_0\mid \forall J\subseteq \{1,\dots,n\},\ |J|\le \tau(\mathcal{B}): 
\\ &\qquad \qquad \qquad \qquad 
 \mathcal{A}_J\subseteq\mathbb{Z}\{m\in \mathcal{A}_J\mid |m|\le d\}\}. 
\end{align*}
\end{corollary} 

\begin{corollary}\label{cor:char p beta sep} 
Assume that $\mathrm{char}(K)=p>0$. Then we have the equality 
\begin{align*} 
\beta_{\mathrm{sep}}(G,V)=
\min\{d\in\mathbb{N}_0\mid &\forall J\subseteq\{1,\dots,n\}\  \exists \alpha_J\in \mathbb{N}_0:\ \\ &p^{\alpha_J}
\mathcal{A}_J\subseteq \mathbb{Z}\{m\in\mathcal{A}_J\mid 
|m|\le d\}\}
\\= \min\{d\in\mathbb{N}_0\mid &\forall J\subseteq\{1,\dots,n\},\ |J|\le\tau_p(\mathcal{B}) 
\  \exists \alpha_J\in \mathbb{N}_0:\ \\ &p^{\alpha_J}
\mathcal{A}_J\subseteq \mathbb{Z}\{m\in\mathcal{A}_J\mid 
|m|\le d\}\}
\end{align*}
\end{corollary}

An $n$-dimensional  representation of an $s$-dimensional torus $(K^\times)^s=K^\times\times\cdots\times K^\times$ 
can be given by an 
$s\times n$ matrix with integer entries. For $A=(a_{ij})_{i=1,\dots,s}^{j=1,\dots,n}$ consider 
$V=V(A)$, where the character $\chi_j$ (the character of the representation of 
$(K^\times)^s$ on the summand $V_j$ in \eqref{eq:decomposition of V}) is 
\[\chi_j:(K^\times )^s\to K^\times, \qquad (z_1,\dots,z_s)\mapsto \prod_{i=1}^s z_i^{a_{ij}}.\]   
The submonoid $\mathcal{B}(\chi_1,\dots,\chi_n)$ is the intersection of $\mathbb{N}_0^n$ and the kernel $\{v\in \mathbb{Z}^n\mid Av=0\in \mathbb{Z}^s\}$ of the matrix $A$. 

\begin{example}\label{example:(K*)^s} 
Consider the $(2s+1)$-dimensional representation $V(A_t)$ of the $s$-dimensional torus $(K^\times)^s$, where for some positive integer $t\ge 2$, we have 
\[A_t=\left(\begin{array}{ccccccccccc}1 & 0 & \dots &  & 0 & 1 & -t & 0 &  \dots &  & 0 \\ 0 & 1&  &  & 0 & 1 & 0 & -t &  &  & 0 \\ 0 & 0 &  &  & 0 & 1 & 0 & 0 &  &  & 0 \\ \vdots & \vdots &  &  & \vdots & \vdots & \vdots & \vdots &   &   & \vdots \\0 & 0 &\dots &  & 1 & 1 & 0 & 0 &  \dots&  & -t\end{array}\right)\in \mathbb{Z}^{s\times (2s+1)}.\]

It is easy to verify that the atoms in the submonoid $\mathcal{B}(\chi_1,\dots,\chi_{2s+1})\subset \mathbb{N}_0^{2s+1}$ are the columns $c_1,\dots,c_{s+t}$ of the following $(2s+1)\times (s+t)$ matrix: 
\[\left(\begin{array}{ccccccccccc}t & 0 & \dots & 0 & 0 & 1 & 2 & 3 & \dots & \dots & t-1 
\\0 & t &   & 0 & 0 & 1 & 2 & 3 &   &  & t-1 
\\\vdots &   &   & \vdots & \vdots & \vdots & \vdots & \vdots &   &   & \vdots 
\\0 & 0 &   & t & 0 & 1 & 2 & 3 &   &   & t-1 
\\0 & 0 &   & 0 & t & t-1 & t-2 & t-3 &   &   & 1 
\\ 1& 0 &  & 0 & 1&1&1&1&&&1
\\0 & 1 &   & 0 & 1 & 1 & 1 & 1 &   &   & 1 
\\ \vdots &  &   &  \vdots & \vdots & \vdots & \vdots & \vdots &   &   & \vdots 
\\0 & 0 & \dots & 1 & 1 & 1 & 1 & 1 & \dots & \dots & 1\end{array}\right)\]
The atoms in $\mathcal{B}(\chi_1,\dots,\chi_{2s+1})$ whose support is a proper subset 
of $\{1,\dots,2s+1\}$ are $c_1,\dots,c_s,c_{s+1}$. We have 
$|c_1|=\cdots=|c_s|=t+1$ and $|c_{s+1}|=t+s$. It follows that for all proper subset $J$ of 
$\{1,\dots,2s+1\}$, the monoid $\mathcal{B}_J$ is generated by $\{c_i\mid i\in\{1,\dots,s+1\},\ \mathrm{supp}(c_i)\subseteq J\}$. 
For $j=1,2,\dots,t-1$ we have the equality 
\[c_1+\cdots +c_s+(t-j)(c_{s+1}-c_{s+2})=c_{s+1+j}.\] 
Therefore all atoms of $\mathcal{B}$ are contained in 
$\mathbb{Z}\{c_1,\dots,c_s,c_{s+1},c_{s+2}\}$. 
Summarizing, the above considerations imply that setting $
M:=\{c_1,\dots,c_s,c_{s+1},c_{s+2}\}$ we have that 
for all $I\subseteq \{1,\dots,2s+1\}$, $\mathcal{A}_I\subset \mathbb{Z} M_J$. 
Consequently, by Theorem~\ref{thm:main} (or Theorem~\ref{thm:pos char main} when $\mathrm{char}(K)=p>0$) the following is a separating set in 
$\mathcal{O}(V(A_t))^{(K^\times)^s}$: 
\[S:=\{x_1^tx_{s+2}, \ x_2^tx_{s+3},\  \dots \ , x_s^tx_{2s+1}, \   x_{s+1}^tx_{s+2}\cdots x_{2s+1}, \  x_1\cdots x_s x_{s+1}^{t-1}x_{s+2}\cdots x_{2s+1}\}.\]
On the other hand, by Proposition~\ref{prop:monoid generators}, a minimal generating system of $\mathcal{O}(V(A_t))^{(K^\times)^2}$ is 
\[S\bigcup \{x_1^j\cdots x_s^jx_{s+1}^{t-j}x_{s+2}\cdots x_{2s+1}\mid j=2,3,\dots,t-1\}.\] 
In particular, we have 
\[\beta((K^\times)^s,V(A_t))=st+1 \qquad \text{ and } \qquad 
\beta_{\mathrm{sep}}((K^\times)^s,V(A_t))\le t+2s-1.\] 
\end{example}

\begin{corollary} \label{cor:inf}
If $\dim(G)>0$, then 
\[\inf_V\left\{\frac{\beta_{\mathrm{sep}}(G,V)}{\beta(G,V)}
\right\}\le \frac{1}{\dim(G)}.\] 
where the infimum is taken over all finite dimensional representations of $G$. 
\end{corollary} 

\begin{proof} 
By \cite[8.7 Proposition]{borel} we have 
$G\cong G^\circ\times A$ where $A$ is a finite abelian group and $G^\circ$ is a torus of rank 
$s:=\dim(G)$. It follows that we have surjective homomorphism $G\to (K^\times)^s$ of algebraic groups, and so the representation $V(A_t)$ constructed in 
Example~\ref{example:(K*)^s} lifts to a representation  of $G$ with the same algebra of invariants. Consequently, the limit of the ratio 
$\frac{\beta_{\mathrm{sep}}(G,V(A_t))}{\beta(G,V(A_t))}\le \frac{t+2s-1}{st+1}$ 
and $\lim_{t\to\infty}  \frac{t+2s-1}{st+1}=\frac 1s$. 
\end{proof} 

\begin{example}
Consider the representation $V(A)$ of $K^\times$ where 
$A=[a_1,\dots,a_n]$, $a_1\ge a_2\ge \cdots \ge a_n$, $a_1>0$, $a_n<0$, and 
$\mathrm{gcd}(a_1,a_n)=1$.  Then $\mathcal{B}_{\{1,n\}}$ is generated (as a monoid) by 
$[-a_n,0,\dots,0,a_1]^T$, whence $\beta_{\mathrm{sep}}(\mathcal{O}(V(A))^{K^\times}\ge a_1-a_n$. On the other hand, $\beta(K^\times,V(A))=a_1-a_n$ by  
\cite[Theorem 1]{wehlau}. Consequently, we have the equality 
$\beta_{\mathrm{sep}}(K^\times,V(A))=\beta(K^\times,V(A))$. 
\end{example}


\section{A generalization of Theorem~\ref{thm:main}} \label{sec:generalization}

Throughout this section $D$ stands for an arbitrary subset of $\mathbb{N}_0^n$, and  
$R:=K[x^d\mid d\in D]$ denotes the subalgebra 
 of the polynomial algebra $\mathcal{O}(V)=K[x_1,\dots,x_n]$ generated by the monomials $x^d$ ($d\in D$). A subset $S\subseteq R$ is \emph{separating} if for any $v,w\in V$, such that $f(v)\neq f(w)$ for some $f\in R$, there exists $h\in S$ with $h(v)\neq h(w)$. 
 Note that any separating set $S$ contains a finite separating subset: this follows from the fact that $K[x_1,\dots,x_n]$ is noetherian by (a straightforward modification of) the argument in the proof of \cite[Theorem 2.4.8]{derksen-kemper}. 

\begin{theorem} \label{thm:general} Assume the $\mathrm{char}(K)=0$, Then the following conditions are equivalent for a subset $M$ of $D$: 
\begin{itemize}
\item[(i)] The monomials $\{x^m\mid m\in M\}$ form a separating set in $R$. 
\item[(ii)] For any subset $J\subseteq\{1,\dots,n\}$, $D$ is contained in $\mathbb{Z}M_J$. 
\end{itemize}
\end{theorem} 

The proof of (i) $\Longrightarrow$ (ii) is based on the following algebro-geometric lemma, which is a special case of \cite[18.2 Proposition, p. 43]{borel}: 

\begin{lemma}\label{lemma:constant along fibers} 
Let $\varphi:V\to Y$ be a dominant morphism from $V$ to an affine variety $Y$, and assume that 
$h\in \mathcal{O}(V)$ is constant along the fibres of $\varphi$. Then $h$ belongs to the purely inseparable closure of $\varphi^*(K(Y))$, where $K(Y)$ is the rational function field of $Y$ and $\varphi^*$ is the comorphism of $\varphi$.  
\end{lemma} 

For $J\subseteq \{1,\dots,n\}$ set $R_J:=K[x^d\mid d\in D_J]$. 
Lemma~\ref{lemma:2} has the following generalization: 

\begin{lemma} \label{lemma:R_J} 
If $\{x^m\mid m\in M\}$ is a separating set in $R$ for some $M\subseteq D$, then 
$\{x^m\mid m\in M_J\}$ is a separating set in $R_J$ for any $J\subseteq \{1,\dots,n\}$. 
\end{lemma} 

\begin{proof}
Suppose that $v,w\in V$ can be separated by $R_J$, so there is some 
$d\in D_J$ with $x^d(v)\neq x^d(w)$. Clearly $x^d(v)=x^d(v_J)$ and $x^d(w)=x^d(w_J)$. 
So $v_J$ and $w_J$ can be separated by $R$. By assumption, there exists an $m\in M$ with $x^m(v_J)\neq x^m(w_J)$. It follows that $\mathrm{supp}(m)\subseteq J$, since otherwise $x^m(v_J)=0=x^m(w_J)$ would lead to a contradiction. Thus $m\in M_J$, 
and $x^m(v)=x^m(v_J)\neq x^m(w_J)=x^m(w)$ shows that $v$ and $w$ can be separated by $M_J$. 
\end{proof} 

\begin{proofof}{Theorem~\ref{thm:general}} (i) $\Longrightarrow$ (ii): 
By Lemma~\ref{lemma:R_J} it is sufficient to show that if 
$\{x^m\mid m\in M\}$ is a separating set in $R$ for some subset $M$ of $D$, then 
$D\subseteq \mathbb{Z}M$. As we mentioned in the first paragraph of this section, 
$M$ contains a finite subset $\{m_1,\dots,m_t\}$ such that $\{x^{m_i}\mid i=1,\dots,t\}$ is a separating set in $R$. 

Consider the map $\varphi:V\to K^t$ whose coordinate functions are the $x^{m_i}$,  
$i=1,\dots,t$. Denote by $Y$ the Zariski closure of $\varphi(V)$ in $K^t$. Then $\varphi:V\to Y$ is a dominant morphism of affine varieties. Take an arbitrary non-zero $d\in D$. The assumption that the coordinate functions of $\varphi$ form a separating set in $R$ implies that $x^d$ is constant along the fibres of $\varphi$. Therefore by Lemma~\ref{lemma:constant along fibers}, there exist polynomials $h_1,h_2$ in $t$ variables, such that $h_2(x^{m_1},\dots,x^{m_t})\neq 0\in \mathcal{O}(V)$ and 
\[x^d=\frac{h_1(x^{m_1},\dots,x^{m_t})}{h_2(x^{m_1},\dots,x^{m_t})}\] 
in the field of fractions of $\mathcal{O}(V)$, 
implying in turn 
\[0\neq x^dh_2(x^{m_1},\dots,x^{m_t})=h_1(x^{m_1},\dots,x^{m_t})\in \mathcal{O}(V).\]  
This clearly implies the existence of an equality of the form 
\[x^d(x^{m_1})^{a_1}\cdots (x^{m_t})^{a_t}=(x^{m_1})^{b_1}\cdots (x^{m_t})^{b_t},\] 
where $a_1,\dots,a_t,b_1,\dots,b_t\in \mathbb{N}_0$. 
We infer $d=\sum_{i=1}^t(b_i-a_i)m_i\in \mathbb{Z}M$. 

(ii) $\Longrightarrow$ (i):  The proof of Lemma~\ref{lemma:4} works verbatim, with $\mathcal{B}$ replaced by $D$. 
\end{proofof} 

In positive characteristic the argument in the above proof 
yields the following 
(the reference to Lemma~\ref{lemma:4} has to be replaced by a reference to 
Lemma~\ref{lemma:p4}): 

\begin{theorem} \label{thm:general pos char} Assume the $\mathrm{char}(K)=p>0$, Then the following conditions are equivalent for a subset $M$ of $D$: 
\begin{itemize}
\item[(i)] The monomials $\{x^m\mid m\in M\}$ form a separating set in $R$. 
\item[(ii)] For any subset $J\subseteq\{1,\dots,n\}$ and any $d\in D$ there is a non-negative integer $\alpha$ such that  $p^{\alpha}d\in \mathbb{Z}M_J$. 
\end{itemize}
\end{theorem} 


\end{document}